\documentclass[12pt,reqno]{amsart}
\usepackage{amsmath,amsthm,amsfonts,amssymb}
\usepackage{graphicx}
\usepackage[all]{xy}

\newtheorem{theorem}{Theorem}
\newtheorem{lemma}[theorem]{Lemma}

\newtheorem{proposition}[theorem]{Proposition}

\numberwithin{equation}{section}

\begin{document}

\title[Boundedness of the twisted paraproduct]{Boundedness of the twisted paraproduct}
\author{Vjekoslav Kova\v{c}}
\address{Vjekoslav Kova\v{c}, Department of Mathematics, UCLA, Los Angeles, CA 90095-1555, vjekovac@math.ucla.edu}

\begin{abstract}
We prove $\mathrm{L}^p$ estimates for a two-dimensional bilinear operator of paraproduct type.
This result answers a question posed by Demeter and Thiele in \cite{DT}.
\end{abstract}

\maketitle

\section{Introduction and overview of results}
\label{vksectionintro}

Let us denote dyadic martingale averages and differences by
$$ \mathbb{E}_{k}f := \sum_{|I|=2^{-k}} \!\big({\textstyle\frac{1}{|I|}\int_{I}f}\big)\,\mathbf{1}_{I}\,,
\qquad \Delta_{k}f := \mathbb{E}_{k+1}f - \mathbb{E}_{k}f \,, $$
for every $k\in\mathbb{Z}$, where the sum is taken over dyadic intervals $I\subseteq\mathbb{R}$ of length $2^{-k}$.
When we apply an operator in only one variable of a two-dimensional function,
we mark it with that variable in the superscript. For instance,
$$ (\mathbb{E}_{k}^{(1)}F)(x,y) := \big(\mathbb{E}_{k}F(\cdot,y)\big)(x) \,. $$
The \emph{dyadic twisted paraproduct} is defined as
\begin{equation}
\label{vkeqparaproductdyadic}
T_{\mathrm{d}}(F,G) := \sum_{k\in\mathbb{Z}}\, (\mathbb{E}_{k}^{(1)}F) (\Delta_{k}^{(2)}G) \,.
\end{equation}
In the continuous case, let $\mathrm{P}_{\varphi}$ denote the Fourier multiplier with symbol $\hat{\varphi}$, i.e.
$$ \mathrm{P}_{\varphi}f := f \ast \varphi \,. $$
Take two functions $\varphi,\psi\in\mathrm{C}^1(\mathbb{R})$ satisfying\footnote{For two nonnegative quantities
$A$ and $B$, we write $A\lesssim B$ if there exists an absolute constant $C\geq 0$ such that $A\leq C B$,
and we write $A\lesssim_P B$ if $A\leq C_P B$ holds for some constant $C_P\geq 0$ depending on a parameter $P$.
Finally, we write $A\sim_P B$ if both $A\lesssim_P B$ and $B\lesssim_P A$.}
\begin{equation}
\label{vkeqsymbolbounds}
|\partial^{j} \varphi(x)| \lesssim (1+|x|)^{-3},\quad
|\partial^{j} \psi(x)| \lesssim (1+|x|)^{-3},\quad
\textrm{for }j=0,1 \,,
\end{equation}
and
$$ \mathrm{supp}(\hat{\psi}) \subseteq \{\xi\in\mathbb{R} \,:\, {\textstyle\frac{1}{2}}\!\leq\! |\xi|\leq 2\}\,. $$
For every $k\in\mathbb{Z}$ denote
\,$\varphi_k(t) := 2^k \varphi(2^k t)$\, and \,$\psi_k(t):= 2^k \psi(2^k t)$.
The associated \emph{continuous twisted paraproduct} is defined as
\begin{equation}
\label{vkeqparaproductreal}
T_{\mathrm{c}}(F,G) := \sum_{k\in\mathbb{Z}}\, (\mathrm{P}_{\varphi_k}^{(1)}F) (\mathrm{P}_{\psi_k}^{(2)}G) \,.
\end{equation}
We are interested in strong-type estimates
\begin{equation}
\label{vkeqstrongtype}
\| T(F,G) \|_{\mathrm{L}^{pq/(p+q)}(\mathbb{R}^2)} \,\lesssim_{p,q}
\|F\|_{\mathrm{L}^p(\mathbb{R}^2)} \|G\|_{\mathrm{L}^q(\mathbb{R}^2)} \,,
\end{equation}
and weak-type estimates
\begin{equation}
\label{vkeqweaktype}
\alpha\ \big| \big\{ (x,y)\in\mathbb{R}^2 :\, |T(F,G)(x,y)|>\alpha \big\} \big|^{(p+q)/pq}
\,\lesssim_{p,q} \|F\|_{\mathrm{L}^p(\mathbb{R}^2)} \|G\|_{\mathrm{L}^q(\mathbb{R}^2)}
\end{equation}
for (\ref{vkeqparaproductdyadic}) and (\ref{vkeqparaproductreal}).
The exponent $\frac{pq}{p+q}$ is mandated by scaling invariance.
When $p=\infty$ or $q=\infty$, we interpret it as $q$ or $p$ respectively.

\smallskip
The main result of the paper establishes (\ref{vkeqstrongtype}) and (\ref{vkeqweaktype})
in certain ranges of $(p,q)$.
\begin{theorem}
\label{vktheoremmainparaprod}
\begin{itemize}
\item[(a)]
Operators $T_\mathrm{d}$ and $T_\mathrm{c}$
satisfy the strong bound \emph{(\ref{vkeqstrongtype})} if
$$ {\textstyle 1<p,q<\infty,\ \ \frac{1}{p}+\frac{1}{q}>\frac{1}{2}}\,. $$
\item[(b)]
Additionally, operators $T_\mathrm{d}$ and $T_\mathrm{c}$
satisfy the weak bound \emph{(\ref{vkeqweaktype})} when
$$ p=1,\ 1\leq q<\infty\ \,\textrm{ or }\,\ q=1,\ 1\leq p<\infty \,. $$
\item[(c)]
The weak estimate \emph{(\ref{vkeqweaktype})} fails for \,$p=\infty$\, or \,$q=\infty$\,.
\end{itemize}
\end{theorem}
\begin{figure}[htbp]
\includegraphics[width=0.6\textwidth]{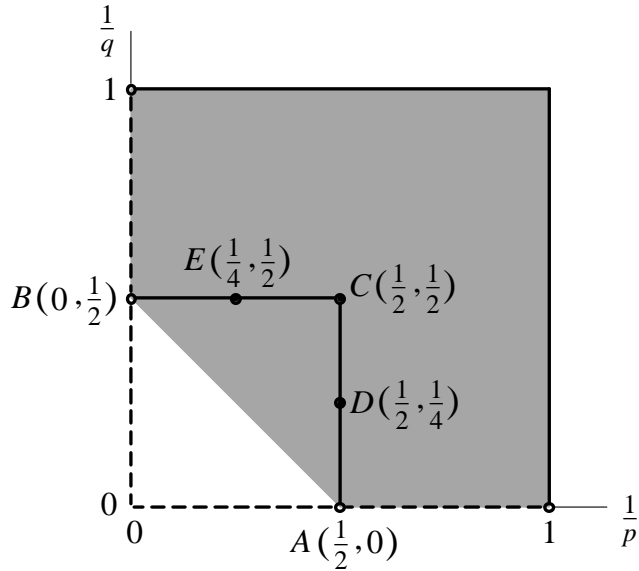}
\caption{The range of exponents we discuss in this paper.}
\label{vkimageexp}
\end{figure}
The name \emph{twisted paraproduct} was suggested by Camil Muscalu
because there is a ``twist'' in the variables in which
the convolutions (or the martingale projections) are performed,
as opposed to the case of the ordinary paraproduct.
No bounds on (\ref{vkeqparaproductdyadic}) or (\ref{vkeqparaproductreal}) were known prior to this work.
A conditional result was shown by Bernicot in \cite{B}, assuming boundedness in some range,
and expanding the range towards lower exponents using a fiber-wise Calder\'{o}n-Zygmund decomposition.
We repeat his argument in the dyadic setting in Section \ref{vksectionextrange},
for the purpose of extending the boundedness region established in Sections \ref{vksectiontelescoping}
and \ref{vksectionsummingtrees}.

Figure \ref{vkimageexp} depicts the range of exponents in Theorem \ref{vktheoremmainparaprod}.
The shaded region satisfies the strong estimate, while for two solid sides of the unit square we only
establish the weak estimates. The two dashed sides of the square represent exponents for which
we show that even the weak estimate fails.
The white triangle in the lower left corner is the region we do not tackle in this paper.

The proof of Theorem \ref{vktheoremmainparaprod} is organized as follows.
Sections \ref{vksectiontelescoping} and \ref{vksectionsummingtrees}
prove estimates for $T_\mathrm{d}$ in the interior of triangle $ABC$.
In Section \ref{vksectionextrange} the rest of bounds for $T_\mathrm{d}$ are obtained.
Section \ref{vksectionrealcase} establishes bounds for $T_\mathrm{c}$ by relating $T_\mathrm{c}$ to $T_\mathrm{d}$.
Finally, in Section \ref{vksectioncounterex} we discuss the counterexamples.
In the closing section we sketch a simpler proof
for points $D$ and $E$ only.

\medskip
\noindent
\textbf{Several remarks.}
Before going into the proofs, we make several simple observations about $T_\mathrm{d}$.
Note that Theorem \ref{vktheoremmainparaprod} also gives estimates for a family of shifted operators
$$ (F,G) \mapsto \sum_{k\in\mathbb{Z}}\, (\mathbb{E}_{k+k_0}^{(1)}F) (\Delta_{k}^{(2)}G) $$
uniformly in $k_0\in\mathbb{Z}$, because the last sum can be rewritten as
$$ \mathrm{D}_{(2^{-k_0},1)} \,T_{\mathrm{d}}\big(\mathrm{D}_{(2^{k_0},1)}F,\, \mathrm{D}_{(2^{k_0},1)}G\big) \,. $$
Here $\mathrm{D}_{(a,1)}$ denotes the non-isotropic dilation\,
$(\mathrm{D}_{(a,1)}F)(x,y):=F(a^{-1}x,y)$.

If $F$ and $G$ are (say) compactly supported, then one can write
\begin{equation}
\label{vkeqsymmetryintro}
T_{\mathrm{d}}(F,G) = FG - \sum_{k\in\mathbb{Z}}\, (\Delta_{k}^{(1)}F) (\mathbb{E}_{k+1}^{(2)}G) \,.
\end{equation}
Combining this with the previous remark and the fact that the pointwise product $FG$ satisfies
H\"{o}lder's inequality,
we see that the set of estimates for $T_\mathrm{d}(F,G)$
is indeed symmetric under interchanging $p$ and $q$, $F$ and $G$.
We use this fact to shorten some of the exposition below.

Furthermore, Theorem \ref{vktheoremmainparaprod} implies bounds on more general dyadic operators
of the following type:
\begin{equation}
\label{vkeqsigns}
\Big\|\sum_{k\in\mathbb{Z}} c_k (\mathbb{E}_{k}^{(1)}F)
(\Delta_{k}^{(2)}G)\Big\|_{\mathrm{L}^{pq/(p+q)}} \,\lesssim_{p,q}
\|F\|_{\mathrm{L}^p} \|G\|_{\mathrm{L}^q} \,,
\end{equation}
for any numbers $c_k$ such that $|c_k|\leq 1$.
Here we restrict ourselves to the interior range \,$1<p,q<\infty$, \,$\frac{1}{p}+\frac{1}{q}>\frac{1}{2}$.\,
One simply uses the known bound for $T_\mathrm{d}(F,\widetilde{G})$ with
\,$\widetilde{G} := \sum_{k\in\mathbb{Z}} c_k \,\Delta_k^{(2)} G$,\,
and the dyadic Littlewood-Paley inequality in the second variable.
Note that the flexibility of having coefficients $c_k$ is implicit in the definition
of $T_\mathrm{c}$, and indeed we will repeat a continuous variant of this argument in Section
\ref{vksectionrealcase}.

\medskip
\noindent
\textbf{Some motivation.}
The one-dimensional bilinear Hilbert transform is an object that motivated
most of the modern multilinear time-frequency analysis.
Lacey and Thiele established its boundedness (in a certain range)
in a pair of breakthrough papers \cite{LT1},\cite{LT2}.
Recently, Demeter and Thiele investigated its two-dimensional analogue in \cite{DT}.
For any two linear maps $A,B\colon\mathbb{R}^2\to\mathbb{R}^2$ they considered
$$ T_{A,B}(F,G)(x,y) := \mathrm{p.v.}\int_{\mathbb{R}^2}
F\big((x,y)+A(s,t)\big) \, G\big((x,y)+B(s,t)\big) \, K(s,t) \,ds dt\,, $$
where $K\colon\mathbb{R}^2\setminus\{0,0\}\to\mathbb{C}$ is a Calder\'{o}n-Zygmund kernel,
i.e.\@ $\hat{K}$ is a symbol satisfying
\begin{equation}
\label{vkeqkernel}
|\partial^\alpha \hat{K}(\xi,\eta)| \,\lesssim_\alpha (\xi^2+\eta^2)^{-|\alpha|/2} \,,
\end{equation}
for all derivatives $\partial^\alpha$ up to some large unspecified order.
In \cite{DT}, the bound
$$ \|T_{A,B}(F,G)\|_{\mathrm{L}^{pq/(p+q)}(\mathbb{R}^2)} \, \lesssim_{A,B,p,q} \,
\|F\|_{\mathrm{L}^p(\mathbb{R}^2)} \|G\|_{\mathrm{L}^q(\mathbb{R}^2)} $$
is proved in the range \,$2<p,q<\infty$, \,$\frac{1}{p}+\frac{1}{q}>\frac{1}{2}$,\, and
for most cases depending on $A$ and $B$.

Some instances of $A,B$ can be handled by an adaptation of the approach from \cite{LT1},\cite{LT2},
while some cases lead the authors of \cite{DT} to invent a ``one-and-a-half-dimensional''
time-frequency analysis.
On the other extreme, some instances of $A,B$ degenerate to the one-dimensional bilinear
Hilbert transform or the pointwise product.
Up to the symmetry obtained by considering the adjoints, the only case of $A,B$ that is left unresolved in \cite{DT} is
\begin{equation}
\label{vkeqcase6op}
T(F,G)(x,y) := \mathrm{p.v.}\int_{\mathbb{R}^2} F(x-s,y) \, G(x,y-t) \, K(s,t) \,ds dt \,.
\end{equation}
This case was denoted ``Case 6'', and as remarked there, it is largely degenerate but still nontrivial,
so the usual wave-packet decompositions showed to be ineffective.
It can also be viewed as the simplest example of higher-dimensional phenomena,
i.e.\@ complications not visible from the perspective of multilinear analysis arising in \cite{LT1},\cite{LT2},
and even in quite general framework such as the one in \cite{MTT} or \cite{DPT}.

\smallskip
Theorem \ref{vktheoremmainparaprod} establishes bounds on the \emph{twisted bilinear multiplier}
(\ref{vkeqcase6op}) for the special case of the symbol
$$ \hat{K}(\xi,\eta) = \sum_{k\in\mathbb{Z}} \, \hat{\varphi}(2^{-k}\xi) \,\hat{\psi}(2^{-k}\eta) \,, $$
i.e.\@ the kernel
$$ K(s,t) = \sum_{k\in\mathbb{Z}} \, 2^k\varphi(2^k s) \,2^k\psi(2^k t) \,, $$
with $\varphi$ and $\psi$ as in the introduction.
A standard technique of ``cone decomposition'' (see \cite{T2}) then addresses general kernels $K$.

\smallskip
Our approach is to first work with the dyadic variant (\ref{vkeqparaproductdyadic}),
and then use the square function introduced by Jones, Seeger, and Wright in \cite{JSW}
to transfer to the continuous case.
Also, we dualize and prefer to consider the corresponding trilinear form
$$ \Lambda_{\mathrm{d}}(F,G,H) := \int_{\mathbb{R}^2} T_{\mathrm{d}}(F,G)(x,y) H(x,y) \,dx dy \,. $$
Another reason why we call this object the twisted paraproduct
is because the functions $F,G,H$ are entwined in a way that
the trilinear form $\Lambda_\mathrm{d}$ does not split naturally into
wavelet coefficients of each function separately, as it does for the ordinary paraproduct.
As a substitute we introduce forms encoding ``entwined wavelet coefficients'',
reminiscent of the Gowers box-norm, which plays an important role in the proof.
These forms keep functions intertwined, and we never attempt to break them
but rather exploit their symmetries in an ``induction on scales'' type of argument.

A difference from the classical theory is that we gradually separate functions $F,G,H$
by repeated applications of the Cauchy-Schwarz inequality and a sort of telescoping identity
that switches between the two variables.
This is opposed to the usual approach to the ordinary paraproduct
(even in the multiparameter case \cite{MPTT1},\cite{MPTT2}),
where the Cauchy-Schwarz inequality is applied at once, and it immediately
splits the form into governing operators like maximal and square functions (or their hybrids).
This dominating procedure requires four steps for $\Lambda_{\mathrm{d}}$,
and generally finitely many steps for ``more entwined'' forms in higher-dimensions, which are
very briefly discussed in the closing section.

\smallskip
There seems to be many other higher-dimensional phenomena worth studying.
Another interesting two-dimensional object, more singular than the twisted paraproduct is
$$ \mathrm{p.v.}\int_{\mathbb{R}} F(x-t,y) \, G(x,y-t) \, \frac{dt}{t} \,. $$
Its boundedness is still an open problem.
One also has to notice that the yet more singular bi-parameter bilinear Hilbert transform
$$ \mathrm{p.v.}\int_{\mathbb{R}^2} F(x-s,y-t) \, G(x+s,y+t) \, \frac{ds}{s} \frac{dt}{t} $$
does not satisfy any $\mathrm{L}^p$ estimates, as shown in \cite{MPTT1}.

\medskip
\noindent
\textbf{Acknowledgement.}
The author would like to thank his faculty advisor Prof.\@ Christoph Thiele
for introducing him to the problem, and for his numerous suggestions on how to improve the exposition.
This and related work would not be possible without his constant support and encouragement.

\section{A few words on the the notation}
\label{vksectionnotation}

A \emph{dyadic interval} is an interval of the form $[2^k l,2^k (l+1))$, for some integers $k$ and $l$.
For each dyadic interval $I$, we denote its left and right halves respectively by $I_\mathrm{left}$ and $I_\mathrm{right}$.
\emph{Dyadic squares} and \emph{dyadic rectangles} in $\mathbb{R}^2$ are defined in the obvious way.
For any dyadic interval $I$, denote the Haar scaling function $\varphi^{\mathrm{d}}_I:=|I|^{-1/2}\mathbf{1}_{I}$
and the Haar wavelet $\psi^{\mathrm{d}}_I:=|I|^{-1/2}(\mathbf{1}_{I_\mathrm{left}}-\mathbf{1}_{I_\mathrm{right}})$.
Martingale averages and differences can be alternatively written in the Haar basis:
$$ \mathbb{E}_{k}f = \sum_{|I|=2^{-k}} \!\big({\textstyle\int_{\mathbb{R}}f\varphi^{\mathrm{d}}_I}\big)\,\varphi^{\mathrm{d}}_I \,,\quad\
\Delta_{k}f = \sum_{|I|=2^{-k}} \!\big({\textstyle\int_{\mathbb{R}}f\psi^{\mathrm{d}}_I}\big)\,\psi^{\mathrm{d}}_I \,. $$
In $\mathbb{R}^2$, every dyadic square $Q$ partitions into four congruent dyadic squares that are called
\emph{children} of $Q$, and conversely, $Q$ is said to be their \emph{parent}.

In all of the following except in Section \ref{vksectionrealcase},
the considered functions are assumed to be \emph{nonnegative} dyadic step functions,
i.e.\@ positive finite linear combinations of characteristic functions of dyadic rectangles.
This reduction is enabled by splitting into positive and negative, real and imaginary parts,
and invoking density arguments.

Let $\mathcal{C}$ denote the collection of all dyadic squares in $\mathbb{R}^2$.
Note that $T_{\mathrm{d}}$ and $\Lambda_{\mathrm{d}}$ can be rewritten as
sums over $\mathcal{C}$:
\begin{align*}
& T_{\mathrm{d}}(F,G)(x,y) = \sum_{I\times J\in\mathcal{C}}
\int_{\mathbb{R}^2}\! F(u,y) G(x,v) \ \varphi^{\mathrm{d}}_I(u)\varphi^{\mathrm{d}}_I(x)
\psi^{\mathrm{d}}_J(v)\psi^{\mathrm{d}}_J(y) \ du dv \,, \\
& \Lambda_{\mathrm{d}}(F,G,H) = \sum_{I\times J\in\mathcal{C}}
\int_{\mathbb{R}^4}\! F(u,y) G(x,v) H(x,y) \ \varphi^{\mathrm{d}}_I(u)\varphi^{\mathrm{d}}_I(x)
\psi^{\mathrm{d}}_J(v)\psi^{\mathrm{d}}_J(y) \ du dx dv dy \,.
\end{align*}

\medskip
In the subsequent discussion we will use one notion from additive combinatorics,
namely the \emph{Gowers box norm}.
It is a two-dimensional variant of a series of norms introduced by Gowers in \cite{G1},\cite{G2}
to give quantitative bounds on Szemer\'{e}di's theorem,
and was used by Shkredov in \cite{S} to give bounds on sizes of sets that do not
contain two-dimensional corners.
Its occurrence in \cite{T} is the one we find the most influential.

For any dyadic square $Q=I\times J$
we first define the \emph{Gowers box inner-product} of four functions $F_1,F_2,F_3,F_4$ as
$$ [F_1,F_2,F_3,F_4]_{\Box(Q)}
:= \frac{1}{|Q|^2}\int_{\!I}\!\int_{\!I}\!\int_{\!J}\!\int_{\!J} F_1(u,v) F_2(x,v) F_3(u,y) F_4(x,y) \,du dx dv dy \,. $$
Then for any function $F$ we introduce the \emph{Gowers box norm} as\footnote{If $F(x,y)$ restricted to $Q$
is discretized and viewed as a matrix, then $\|F\|_{\Box(Q)}$ can be recognized as its
(properly normalized) Schatten $4$-norm, i.e.\@ $\ell^4$ norm of the sequence of its singular values.
This comment gives yet one more immediate proof of inequality (\ref{vkeqboxlessl2}) below.}
$$ \|F\|_{\Box(Q)} := [F,F,F,F]_{\Box(Q)}^{1/4} . $$
It is easy to prove the \emph{box Cauchy-Schwarz inequality}:
\begin{equation}
\label{vkeqboxcsb}
[F_1,F_2,F_3,F_4]_{\Box(Q)} \leq \|F_1\|_{\Box(Q)} \|F_2\|_{\Box(Q)} \|F_3\|_{\Box(Q)} \|F_4\|_{\Box(Q)} \,.
\end{equation}
To see (\ref{vkeqboxcsb}), one has to write $|Q|^2\, [F_1,F_2,F_3,F_4]_{\Box(Q)}$ as
$$ \int_{I}\int_{I} \Big(\int_{J} F_1(u,v) F_2(x,v) dv \Big)
\Big( \int_{J} F_3(u,y) F_4(x,y) dy \Big) \,du dx \,, $$
and apply the ordinary Cauchy-Schwarz inequality in $u,x\in I$.
Then one rewrites the result as
\begin{align*}
& \bigg( \int_{J}\int_{J} \Big(\int_{I} F_1(u,v) F_1(u,y) du \Big)
\Big( \int_{I} F_2(x,v) F_2(x,y) dx \Big) \,dv dy \bigg)^{\frac{1}{2}} \\
& \cdot \bigg( \int_{J}\int_{J} \Big(\int_{I} F_3(u,v) F_3(u,y) du \Big)
\Big( \int_{I} F_4(x,v) F_4(x,y) dx \Big) \,dv dy \bigg)^{\frac{1}{2}} \,,
\end{align*}
and applies the Cauchy-Schwarz inequality again, this time in $v,y\in J$.
From here it is also easily seen that $\|\cdot\|_{\Box(Q)}$ is really a norm on functions
supported on $Q$.
On the other hand, a straightforward application of the (ordinary) Cauchy-Schwarz inequality yields
\begin{equation}
\label{vkeqboxlessl2}
\|F\|_{\Box(Q)} \leq \Big(\frac{1}{|Q|}\int_{Q}|F|^2\Big)^{1/2} .
\end{equation}
An alternative way to verify (\ref{vkeqboxlessl2}) is to notice that it is a special case
of the strong $(\frac{1}{2},\frac{1}{2},\frac{1}{2},\frac{1}{2})$
estimate for the quadrilinear form
$$ (F_1,F_2,F_3,F_4)\mapsto|Q|^2[F_1,F_2,F_3,F_4]_{\Box(Q)} \,. $$
Since $(\frac{1}{2},\frac{1}{2},\frac{1}{2},\frac{1}{2})$
is in the convex hull of $(1,0,0,1)$ and $(0,1,1,0)$, we can use complex interpolation, and
it is enough to verify strong type estimates for the latter points, which is trivial.

\section{Telescoping identities over trees}
\label{vksectiontelescoping}

A \emph{tree} is a collection $\mathcal{T}$ of dyadic squares in $\mathbb{R}^2$ such that there exists $Q_\mathcal{T}\in\mathcal{T}$,
called the \emph{root} of $\mathcal{T}$, satisfying $Q\subseteq Q_\mathcal{T}$ for every $Q\in\mathcal{T}$.
A tree $\mathcal{T}$ is said to be \emph{convex} if whenever $Q_1\subseteq Q_2\subseteq Q_3$, and $Q_1,Q_3\in\mathcal{T}$,
then also $Q_2\in\mathcal{T}$.
We will only be working with finite convex trees.
A \emph{leaf} of $\mathcal{T}$ is a square that is not contained in $\mathcal{T}$, but its parent is.
The family of leaves of $\mathcal{T}$ will be denoted $\mathcal{L}(\mathcal{T})$.
Notice that for every finite convex tree $\mathcal{T}$ squares in $\mathcal{L}(\mathcal{T})$
partition the root $Q_\mathcal{T}$.

For any finite convex tree $\mathcal{T}$ we define the local variant of $\Lambda_\mathrm{d}$
that only sums over the squares in $\mathcal{T}$, i.e.
$$ \Lambda_{\mathcal{T}}(F,G,H) := \!\sum_{I\times J\in\mathcal{T}}
\int_{\mathbb{R}^4}\! F(u,y) G(x,v) H(x,y) \,
\varphi^{\mathrm{d}}_I(u)\varphi^{\mathrm{d}}_I(x)\psi^{\mathrm{d}}_J(v)\psi^{\mathrm{d}}_J(y) \, du dx dv dy \,. $$
It turns out to be handy to also introduce a slightly more general quadrilinear form
\begin{align*}
\Theta_{\mathcal{T}}^{(2)}(F_1,F_2,F_3,F_4)
:= \sum_{I\times J\in\mathcal{T}}
\int_{\mathbb{R}^4}\! F_1(u,v) F_2(x,v) F_3(u,y) F_4(x,y) \qquad & \\[-2mm]
\varphi^{\mathrm{d}}_I(u)\varphi^{\mathrm{d}}_I(x)\psi^{\mathrm{d}}_J(v)\psi^{\mathrm{d}}_J(y) \ du dv dx dy & \,,
\end{align*}
and its modified counterpart
\begin{align*}
\Theta_{\mathcal{T}}^{(1)}(F_1,F_2,F_3,F_4)
:= \sum_{I\times J\in\mathcal{T}}\sum_{j\in\{\mathrm{left},\mathrm{right}\}}
\int_{\mathbb{R}^4}\! F_1(u,v) F_2(x,v) F_3(u,y) F_4(x,y) & \\[-1mm]
\psi^{\mathrm{d}}_I(u)\psi^{\mathrm{d}}_I(x)\varphi^{\mathrm{d}}_{J_j}(v)\varphi^{\mathrm{d}}_{J_j}(y) \ du dv dx dy & \,.
\end{align*}
Note that in $\Theta_{\mathcal{T}}^{(1)}$ we actually sum over a certain collection of dyadic rectangles whose horizontal side is
twice longer than the vertical one.
This is just a technicality to make the arguments simpler at the cost of losing (geometric) symmetry.
Also observe that $\Lambda_{\mathcal{T}}(F,G,H)$ can be recognized as $\Theta^{(2)}_{\mathcal{T}}(\mathbf{1},G,F,H)$,
where $\mathbf{1}$ is the constant function on $\mathbb{R}^2$.

Let us also denote for any collection $\mathcal{F}$ of dyadic squares:
\begin{align}
\Xi_{\mathcal{F}}(F_1,F_2,F_3,F_4)
:= \sum_{Q\in\mathcal{F}} \,|Q|\, \big[F_1,F_2,F_3,F_4\big]_{\Box(Q)} \,,
\label{vkeqxidef}
\end{align}
or equivalently
\begin{align}
\Xi_{\mathcal{F}}(F_1,F_2,F_3,F_4)
= \sum_{I\times J\in\mathcal{F}}
\int_{\mathbb{R}^4}\! F_1(u,v) F_2(x,v) F_3(u,y) F_4(x,y) \qquad & \nonumber \\[-2mm]
\varphi^{\mathrm{d}}_I(u)\varphi^{\mathrm{d}}_I(x)\varphi^{\mathrm{d}}_J(v)\varphi^{\mathrm{d}}_J(y) \ du dv dx dy & \label{vkeqxialt} \,.
\end{align}

The following lemma is the core of our method.
\begin{lemma}[Telescoping identity]
\label{vklemmatelescoping}
For any finite convex tree $\mathcal{T}$ with root $Q_\mathcal{T}$ we have
\begin{align*}
& \Theta_{\mathcal{T}}^{(1)}(F_1,F_2,F_3,F_4) + \Theta_{\mathcal{T}}^{(2)}(F_1,F_2,F_3,F_4) \\
& = \ \Xi_{\mathcal{L}(\mathcal{T})}(F_1,F_2,F_3,F_4) - \Xi_{\{Q_\mathcal{T}\}}(F_1,F_2,F_3,F_4) \,.
\end{align*}
\end{lemma}
\begin{proof}
We first note that it is enough to verify the identity when $\mathcal{T}$ consists of only one square,
as in general the right hand side can be expanded into a telescoping sum
$$ \sum_{Q\in\mathcal{T}} \bigg(
\sum_{\widetilde{Q}\textrm{ is \!a \!child \!of }Q}
\!\!\!\!\!\!\Xi_{\{\widetilde{Q}\}} \ \, - \ \Xi_{\{Q\}} \bigg) \,. $$
Here is where we use that $\mathcal{T}$ is convex, which means that each square
$Q\in\mathcal{T}\setminus\{Q_\mathcal{T}\}$
has all four children and the parent in $\mathcal{T}\cup\mathcal{L}(\mathcal{T})$.

Second, observe that when $\mathcal{T}$ has only one square $I\times J$, then using (\ref{vkeqxialt}) the identity
reduces to showing
\begin{align}
& \sum_{j\in\{\mathrm{left},\mathrm{right}\}}\!\!\!
\psi^{\mathrm{d}}_I(u)\psi^{\mathrm{d}}_I(x)\varphi^{\mathrm{d}}_{J_j}(v)\varphi^{\mathrm{d}}_{J_j}(y)
\ + \ \varphi^{\mathrm{d}}_I(u)\varphi^{\mathrm{d}}_I(x)\psi^{\mathrm{d}}_{J}(v)\psi^{\mathrm{d}}_{J}(y) \nonumber \\
& = \sum_{i,j\in\{\mathrm{left},\mathrm{right}\}}\!\!\!
\varphi^{\mathrm{d}}_{I_i}(u)\varphi^{\mathrm{d}}_{I_i}(x)\varphi^{\mathrm{d}}_{J_j}(v)\varphi^{\mathrm{d}}_{J_j}(y)
\ - \ \varphi^{\mathrm{d}}_I(u)\varphi^{\mathrm{d}}_I(x)\varphi^{\mathrm{d}}_{J}(v)\varphi^{\mathrm{d}}_{J}(y) \,, \label{vkeqteleproof1}
\end{align}
multiplying by \,$F_1(u,v) F_2(x,v) F_3(u,y) F_4(x,y)$\, and finally integrating.
By adding and subtracting one extra term,
equality (\ref{vkeqteleproof1}) can be further rewritten as
\begin{align}
& \Big(\sum_{j\in\{\mathrm{left},\mathrm{right}\}}\!\!\!\!\!\!\varphi^{\mathrm{d}}_{J_j}(v)\varphi^{\mathrm{d}}_{J_j}(y)\Big)
\bigg( \varphi^{\mathrm{d}}_{I}(u)\varphi^{\mathrm{d}}_{I}(x)+\psi^{\mathrm{d}}_{I}(u)\psi^{\mathrm{d}}_{I}(x)
-\!\!\!\!\sum_{i\in\{\mathrm{left},\mathrm{right}\}}\!\!\!\!\!\!\varphi^{\mathrm{d}}_{I_i}(u)\varphi^{\mathrm{d}}_{I_i}(x) \bigg) \nonumber \\
& + \ \varphi^{\mathrm{d}}_I(u)\varphi^{\mathrm{d}}_I(x) \bigg( \varphi^{\mathrm{d}}_{J}(v)\varphi^{\mathrm{d}}_{J}(y)
+\psi^{\mathrm{d}}_{J}(v)\psi^{\mathrm{d}}_{J}(y)
-\!\!\!\!\sum_{j\in\{\mathrm{left},\mathrm{right}\}}\!\!\!\!\!\!\varphi^{\mathrm{d}}_{J_j}(v)\varphi^{\mathrm{d}}_{J_j}(y) \bigg) = 0 \,.
\label{vkeqteleproof2}
\end{align}
It remains to notice
\begin{align*}
& \varphi^{\mathrm{d}}_I(u)\varphi^{\mathrm{d}}_I(x) + \psi^{\mathrm{d}}_I(u)\psi^{\mathrm{d}}_I(x) \\
& = \ |I|^{-1}\big(\mathbf{1}_{I_\mathrm{left}}(u)+\mathbf{1}_{I_\mathrm{right}}(u)\big)
\big(\mathbf{1}_{I_\mathrm{left}}(x)+\mathbf{1}_{I_\mathrm{right}}(x)\big) \\
& \ \,+ |I|^{-1}\big(\mathbf{1}_{I_\mathrm{left}}(u)-\mathbf{1}_{I_\mathrm{right}}(u)\big)
\big(\mathbf{1}_{I_\mathrm{left}}(x)-\mathbf{1}_{I_\mathrm{right}}(x)\big) \\
& = \ 2|I|^{-1}\mathbf{1}_{I_\mathrm{left}}(u)\mathbf{1}_{I_\mathrm{left}}(x)
+ 2|I|^{-1}\mathbf{1}_{I_\mathrm{right}}(u)\mathbf{1}_{I_\mathrm{right}}(x) \\
& = \varphi^{\mathrm{d}}_{I_\mathrm{left}}(u)\varphi^{\mathrm{d}}_{I_\mathrm{left}}(x)
+ \varphi^{\mathrm{d}}_{I_\mathrm{right}}(u)\varphi^{\mathrm{d}}_{I_\mathrm{right}}(x) \,,
\end{align*}
and analogously
$$ \varphi^{\mathrm{d}}_J(v)\varphi^{\mathrm{d}}_J(y) + \psi^{\mathrm{d}}_J(v)\psi^{\mathrm{d}}_J(y)=
\varphi^{\mathrm{d}}_{J_\mathrm{left}}(v)\varphi^{\mathrm{d}}_{J_\mathrm{left}}(y)
+ \varphi^{\mathrm{d}}_{J_\mathrm{right}}(v)\varphi^{\mathrm{d}}_{J_\mathrm{right}}(y) \,. $$
\end{proof}

Let us remark that, since we assume $F_1,F_2,F_3,F_4\geq 0$, we have
$$ \Xi_{\{Q_\mathcal{T}\}}(F_1,F_2,F_3,F_4)\geq 0\,, $$
so the right hand side of the telescoping identity
is at most $\Xi_{\mathcal{L}(\mathcal{T})}(F_1,F_2,F_3,F_4)$.
We will use this observation without further mention.

\smallskip
The next lemma will be used to gradually control the forms $\Theta_\mathcal{T}^{(1)}$, $\Theta_\mathcal{T}^{(2)}$.
\begin{lemma}[Reduction inequalities]
\label{vklemmareduction}
\begin{align*}
\big|\Theta^{(1)}_{\mathcal{T}}(F_1,F_2,F_3,F_4)\big| & \ \leq
\ \Theta^{(1)}_{\mathcal{T}}(F_1,F_1,F_3,F_3)^{1/2}
\ \Theta^{(1)}_{\mathcal{T}}(F_2,F_2,F_4,F_4)^{1/2} \\
\big|\Theta^{(2)}_{\mathcal{T}}(F_1,F_2,F_3,F_4)\big| & \ \leq
\ \Theta^{(2)}_{\mathcal{T}}(F_1,F_2,F_1,F_2)^{1/2}
\ \Theta^{(2)}_{\mathcal{T}}(F_3,F_4,F_3,F_4)^{1/2}
\end{align*}
\end{lemma}
\begin{proof}
Rewrite $\Theta_{\mathcal{T}}^{(2)}(F_1,F_2,F_3,F_4)$ as
\begin{align*}
& \sum_{I\times J\in\mathcal{T}}
\int_{\mathbb{R}^2} \bigg( \int_{\mathbb{R}} F_1(u,v) F_2(x,v) \psi^{\mathrm{d}}_J(v) dv \bigg) \\
& \qquad\quad \cdot\bigg( \int_{\mathbb{R}} F_3(u,y) F_4(x,y) \psi^{\mathrm{d}}_J(y) dy \bigg)
\,\varphi^{\mathrm{d}}_I(u)\varphi^{\mathrm{d}}_I(x)\, du dx \,,
\end{align*}
and apply the Cauchy-Schwarz inequality,
first over $(u,x)\in I\times I$,
and then over $I\times J\in\mathcal{T}$.
The inequality for $\Theta_{\mathcal{T}}^{(1)}$ is proved similarly.
\end{proof}

\smallskip
Now we are ready to prove a local estimate, which will be ``integrated'' to a global one in the next section.
\begin{proposition}[Single tree estimate]
\label{vkpropsingletree}
For any finite convex tree $\mathcal{T}$ we have
\begin{equation}
\label{vkeqsingletreetheta}
\big|\Theta^{(2)}_{\mathcal{T}}(F_1,F_2,F_3,F_4)\big| \ \leq \
2|Q_\mathcal{T}| \, \prod_{j=1}^{4}
\max_{Q\in\mathcal{L}(\mathcal{T})}\!\left\|F_j\right\|_{\Box(Q)} \,.
\end{equation}
In particular
\begin{equation}
\label{vkeqsingletree}
\big|\Lambda_{\mathcal{T}}(F,G,H)\big| \, \leq \, 2|Q_\mathcal{T}|\,
\big(\max_{Q\in\mathcal{L}(\mathcal{T})}\!\left\|F\right\|_{\Box(Q)}\!\big)
\big(\max_{Q\in\mathcal{L}(\mathcal{T})}\!\left\|G\right\|_{\Box(Q)}\!\big)
\big(\max_{Q\in\mathcal{L}(\mathcal{T})}\!\left\|H\right\|_{\Box(Q)}\!\big) \,.
\end{equation}
\end{proposition}

\begin{proof}
The proof of (\ref{vkeqsingletreetheta}) consists of several alternating applications
of Lemma \ref{vklemmatelescoping} and Lemma \ref{vklemmareduction}.
Start with four non-negative functions\footnote{We have changed the notation in the proof from $F_j$ to $G_j$
to avoid the confusion, since Lemma \ref{vklemmatelescoping} and Lemma \ref{vklemmareduction}
will be applied for various choices of $F_j$.}
$G_1,G_2,G_3,G_4$, and normalize:
\begin{equation}
\label{vkeqnormalization}
\max_{Q\in\mathcal{L}(\mathcal{T})}\!\left\|G_j\right\|_{\Box(Q)} = 1 \,,
\end{equation}
for $j=1,2,3,4$, since the inequality is homogenous.
By scale invariance, we may also assume $|Q_\mathcal{T}|=1$.
Observe that \,$\Theta^{(2)}_{\mathcal{T}}(G_j,G_j,G_j,G_j)\geq 0$,\,
since it can be written as
$$ \sum_{I\times J\in\mathcal{T}}
\int_{\mathbb{R}^2} \bigg( \int_{\mathbb{R}} G_j(u,y) G_j(x,y) \psi^{\mathrm{d}}_J(y) dy \bigg)^2
\varphi^{\mathrm{d}}_I(u) \varphi^{\mathrm{d}}_I(x) \ du dx \,. $$
Thus, from the telescoping identity we get
$$ \Theta^{(1)}_{\mathcal{T}}(G_j,G_j,G_j,G_j) \leq \Xi_{\mathcal{L}(\mathcal{T})}(G_j,G_j,G_j,G_j) \,, $$
and then from (\ref{vkeqnormalization}), (\ref{vkeqxidef}), and the fact that $\mathcal{L}(\mathcal{T})$
partitions $Q_\mathcal{T}$:
$$ \Theta^{(1)}_{\mathcal{T}}(G_j,G_j,G_j,G_j) \leq |Q_\mathcal{T}| = 1 \,. $$
\begin{figure}[t]
{\small $$ \xymatrix @C=-1.78cm {
& & & \Theta_{\mathcal{T}}^{(2)}(G_1,G_2,G_3,G_4) \ar@{->}[dll] \ar@{->}[drr] & & & \\
& \Theta_{\mathcal{T}}^{(2)}(G_1,G_2,G_1,G_2) \ar@{-->}[d] & & & &
\Theta_{\mathcal{T}}^{(2)}(G_3,G_4,G_3,G_4) \ar@{-->}[d] & \\
& \Theta_{\mathcal{T}}^{(1)}(G_1,G_2,G_1,G_2) \ar@{->}[dl] \ar@{->}[dr] & & & &
\Theta_{\mathcal{T}}^{(1)}(G_3,G_4,G_3,G_4) \ar@{->}[dl] \ar@{->}[dr] & \\
\Theta_{\mathcal{T}}^{(1)}(G_1,G_1,G_1,G_1) \ar@{-->}[d] & & \Theta_{\mathcal{T}}^{(1)}(G_2,G_2,G_2,G_2) \ar@{-->}[d] & &
\Theta_{\mathcal{T}}^{(1)}(G_3,G_3,G_3,G_3) \ar@{-->}[d] & & \Theta_{\mathcal{T}}^{(1)}(G_4,G_4,G_4,G_4) \ar@{-->}[d] \\
\Theta_{\mathcal{T}}^{(2)}(G_1,G_1,G_1,G_1) & & \Theta_{\mathcal{T}}^{(2)}(G_2,G_2,G_2,G_2) & &
\Theta_{\mathcal{T}}^{(2)}(G_3,G_3,G_3,G_3) & & \Theta_{\mathcal{T}}^{(2)}(G_4,G_4,G_4,G_4)
} $$ }
\caption{Schematic presentation of the proof of Proposition \ref{vkpropsingletree}.
A solid arrow denotes an application of the reduction inequality,
while a broken arrow denotes an application of the telescoping identity.}
\label{vkdiagram}
\end{figure}
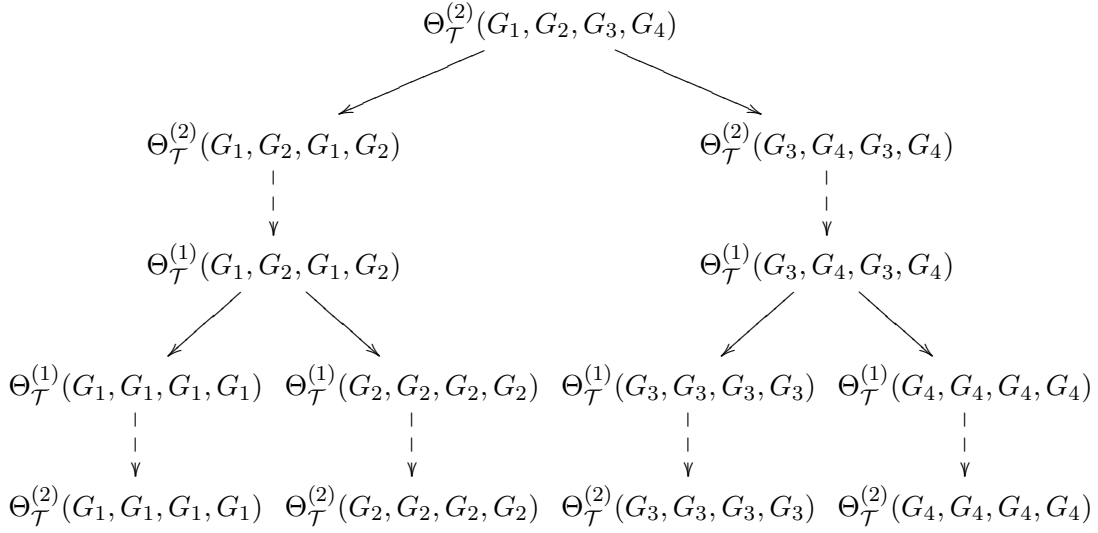
Reduction inequalities now also give
\begin{align}
\big|\Theta^{(1)}_{\mathcal{T}}(G_1,G_2,G_1,G_2)\big| & \leq 1 \,, \label{vkeqtreeproof1} \\
\big|\Theta^{(1)}_{\mathcal{T}}(G_3,G_4,G_3,G_4)\big| & \leq 1 \,. \nonumber
\end{align}
Next, from (\ref{vkeqnormalization}), (\ref{vkeqboxcsb}), and (\ref{vkeqxidef}) one gets
\begin{equation}
\label{vkeqtreeproof2}
\Xi_{\mathcal{L}(\mathcal{T})}(G_1,G_2,G_1,G_2) \leq 1 \,,
\end{equation}
while Lemma \ref{vklemmatelescoping} gives
$$ \Theta^{(2)}_{\mathcal{T}}(G_1,G_2,G_1,G_2) \leq
\Xi_{\mathcal{L}(\mathcal{T})}(G_1,G_2,G_1,G_2) - \Theta^{(1)}_{\mathcal{T}}(G_1,G_2,G_1,G_2) \,, $$
and combining with (\ref{vkeqtreeproof1}), (\ref{vkeqtreeproof2}) yields
$$ \Theta^{(2)}_{\mathcal{T}}(G_1,G_2,G_1,G_2) \leq 2 \,. $$
Completely analogously
$$ \Theta^{(2)}_{\mathcal{T}}(G_3,G_4,G_3,G_4) \leq 2 \,. $$
By another application of Lemma \ref{vklemmareduction}
we end up with
$$ \big|\Theta^{(2)}_{\mathcal{T}}(G_1,G_2,G_3,G_4)\big| \leq 2 \,, $$
and this establishes (\ref{vkeqsingletreetheta}).

For the proof of (\ref{vkeqsingletree}) one has to substitute
$F_1=\mathbf{1}$, $F_2=G$, $F_3=F$, $F_4=H$ into (\ref{vkeqsingletreetheta}).
\end{proof}

The above proof can be represented in the form of a tree diagram, as in Figure \ref{vkdiagram}.
We were inductively bounding terms starting from the bottom and proceeding to the top.
The last row consists of nonnegative terms, allowing us to start the ``induction''.
By every application of the telescoping identity we also get terms with $\Xi_{\mathcal{L}(\mathcal{T})}$,
which we do not denote, and which are controlled by (\ref{vkeqnormalization}) and (\ref{vkeqboxcsb}).

\section{Proving the estimate in the local $\mathrm{L}^2$ case}
\label{vksectionsummingtrees}

In this section we show the bound
\begin{equation}
\label{vkeqcasepqr}
|\Lambda_{\mathrm{d}}(F,G,H)| \ \lesssim_{p,q,r} \|F\|_{\mathrm{L}^{p}} \|G\|_{\mathrm{L}^{q}} \|H\|_{\mathrm{L}^{r}}
\end{equation}
for \ $\frac{1}{p}+\frac{1}{q}+\frac{1}{r}=1$, \ $2<p,q,r<\infty$.
By duality we get (\ref{vkeqstrongtype}) for $T_\mathrm{d}$ in the range
\,$2<p,q<\infty$,\, $\frac{1}{p}+\frac{1}{q}>\frac{1}{2}$.\,
The following material became somewhat standard over the time, and indeed we are closely following the ideas from \cite{T1},
actually in a much simpler setting.

Let us fix dyadic step functions $F,G,H\colon\mathbb{R}^2\to[0,\infty)$, none of them being identically $0$.
To make all arguments finite, in this section we restrict ourselves to considering
only dyadic squares $Q$ satisfying $Q\subseteq[-2^N,2^N)^2$ and $2^{-2N}\leq |Q|\leq 2^{2N}$,
for some (large) fixed positive integer $N$.
Since our bounds will be independent of $N$, letting $N\to\infty$ handles the whole collection $\mathcal{C}$.

We organize the family of dyadic squares in the following way.
For any $k\in\mathbb{Z}$ we define the collection
$$ \mathcal{P}_{k}^{F}
:= \Big\{ Q \ : \ 2^{k} \leq \sup_{Q'\supseteq Q} \|F\|_{\Box(Q')} < 2^{k+1} \Big\} \,, $$
and let $\mathcal{M}_{k}^{F}$ denote the family of maximal squares in $\mathcal{P}_{k}^{F}$
with respect to the set inclusion.
Collections $\mathcal{P}_{k}^{G}$, $\mathcal{M}_{k}^{G}$, $\mathcal{P}_{k}^{H}$, $\mathcal{M}_{k}^{H}$
are defined analogously.
Furthermore, for any triple of integers $k_1,k_2,k_3$ we set
$$ \mathcal{P}_{k_1,k_2,k_3} := \mathcal{P}_{k_1}^{F}\cap\mathcal{P}_{k_2}^{G}\cap\mathcal{P}_{k_3}^{H} \,, $$
and let $\mathcal{M}_{k_1,k_2,k_3}$ denote the family of maximal squares in $\mathcal{P}_{k_1,k_2,k_3}$.

For each $Q\in\mathcal{M}_{k_1,k_2,k_3}$ note that
$$ \mathcal{T}_Q := \{ \widetilde{Q}\in \mathcal{P}_{k_1,k_2,k_3} \, : \, \widetilde{Q}\subseteq Q \} $$
is a convex tree with root $Q$, and that for different $Q$ the corresponding trees $\mathcal{T}_Q$
occupy disjoint regions in $\mathbb{R}^2$.
These trees decompose the collection $\mathcal{P}_{k_1,k_2,k_3}$, for each individual choice of $k_1,k_2,k_3$.

We apply Proposition \ref{vkpropsingletree} to each of the trees $\mathcal{T}_Q$.
Consider any leaf $\widetilde{Q}\in\mathcal{L}(\mathcal{T}_Q)$, and denote its parent by $Q'$.
From $Q'\in\mathcal{T}_Q\subseteq\mathcal{P}_{k_1,k_2,k_3}$ we get
$$ {\textstyle\frac{1}{2}}\|F\|_{\Box(\widetilde{Q})} \leq \|F\|_{\Box(Q')} < 2^{k_1+1} \,, $$
thus $\|F\|_{\Box(\widetilde{Q})} \lesssim 2^{k_1}$,\,
and similarly $\|G\|_{\Box(\widetilde{Q})} \lesssim 2^{k_2}$, $\|H\|_{\Box(\widetilde{Q})} \lesssim 2^{k_3}$,
so the ``single tree estimate'' (\ref{vkeqsingletree}) implies
$$ \big|\Lambda_{\mathcal{T}_Q}(F,G,H)\big| \ \lesssim \
2^{k_1+k_2+k_3} \,|Q| \,. $$

\medskip
We split $\Lambda_\mathrm{d}$ into a sum of $\Lambda_{\mathcal{T}_Q}$
over all $k_1,k_2,k_3\in\mathbb{Z}$ and all $Q\in\mathcal{M}_{k_1,k_2,k_3}$.
In order to finish the proof of (\ref{vkeqcasepqr}), it remains to show
\begin{equation}
\label{vkeqsumtrees1}
\sum_{k_1,k_2,k_3\in\mathbb{Z}} 2^{k_1+k_2+k_3}
\sum_{Q\in\mathcal{M}_{k_1,k_2,k_3}} |Q| \ \lesssim_{p,q,r}
\|F\|_{\mathrm{L}^{p}}\|G\|_{\mathrm{L}^{q}}\|H\|_{\mathrm{L}^{r}} \,.
\end{equation}
The trick from \cite{T1} is to observe that for any fixed triple $k_1,k_2,k_3\in \mathbb{Z}$,
squares in $\mathcal{M}_{k_1}^{F}$ cover squares in $\mathcal{M}_{k_1,k_2,k_3}$,
and the latter are disjoint.
The same is true for $\mathcal{M}_{k_2}^{G}$ and $\mathcal{M}_{k_3}^{H}$.
Thus, it suffices to prove
\begin{align}
\sum_{k_1,k_2,k_3\in\mathbb{Z}}\!\! 2^{k_1+k_2+k_3}
\min \Big( \sum_{Q\in\mathcal{M}_{k_1}^{F}} \!\!\!|Q|, \sum_{Q\in\mathcal{M}_{k_2}^{G}} \!\!\!|Q|,
\sum_{Q\in\mathcal{M}_{k_3}^{H}} \!\!\!|Q| \Big) \quad & \nonumber \\
\lesssim_{p,q,r} \|F\|_{\mathrm{L}^{p}}\|G\|_{\mathrm{L}^{q}}\|H\|_{\mathrm{L}^{r}} & \,. \label{vkeqsumtrees2}
\end{align}

\smallskip
Consider the following version of the dyadic maximal function
$$ \mathrm{M}_{2} F \,:=\, \sup_{Q\in\mathcal{C}}\,
\Big(\frac{1}{|Q|}\int_Q |F|^{2}\Big)^{1/2} \mathbf{1}_Q \,. $$
For each $Q\in\mathcal{M}_{k}^{F}$, from (\ref{vkeqboxlessl2}) and
$\|F\|_{\Box(Q)} \geq 2^{k}$ we have
$Q\subseteq \{\mathrm{M}_{2}F\geq 2^{k}\}$,
and by disjointness
$$ \sum_{Q\in\mathcal{M}_{k}^{F}} |Q| \,\leq\, |\{\mathrm{M}_{2} F\geq 2^k\}| \,. $$
Also note that
$$ \sum_{k\in\mathbb{Z}}\ 2^{p k} |\{\mathrm{M}_{2} F\geq 2^k\}|
\ \sim_p \|\mathrm{M}_{2} F\|_{\mathrm{L}^{p}}^{p}
\, \lesssim_{p} \|F\|_{\mathrm{L}^{p}}^{p} \,, $$
because $\mathrm{M}_2$ is bounded on $\mathrm{L}^p(\mathbb{R}^2)$ for $2<p<\infty$.
Therefore
\begin{equation}
\label{vkeqsumtrees3}
\sum_{k\in\mathbb{Z}}\, 2^{p k}\!\! \sum_{Q\in\mathcal{M}_{k}^{F}}\!\! |Q|
\ \lesssim_{p} \, \|F\|_{\mathrm{L}^{p}}^{p} \,,
\end{equation}
and completely analogously we get
$$ \sum_{k\in\mathbb{Z}} 2^{q k}\!\! \sum_{Q\in\mathcal{M}_{k}^{G}}\!\! |Q|
\ \lesssim_{q} \|G\|_{\mathrm{L}^{q}}^{q}\,, \quad
\sum_{k\in\mathbb{Z}} 2^{r k}\!\! \sum_{Q\in\mathcal{M}_{k}^{H}}\!\! |Q|
\ \lesssim_{r} \|H\|_{\mathrm{L}^{r}}^{r} \,. $$
A purely algebraic ``integration lemma'' stated and proved in \cite{T1} deduces
(\ref{vkeqsumtrees2}) from these three estimates.
The idea is to split the sum in (\ref{vkeqsumtrees2})
into three parts, depending on which of the numbers
$$ \frac{2^{p k_1}}{\|F\|_{\mathrm{L}^p}^{p}}, \
\frac{2^{q k_2}}{\|G\|_{\mathrm{L}^q}^{q}}, \
\frac{2^{r k_3}}{\|H\|_{\mathrm{L}^r}^{r}} $$
is the largest. For instance, the part of the sum over
$$ {\textstyle S_1 := \{(k_1,k_2,k_3)\in\mathbb{Z}^3:\,
\frac{2^{p k_1}}{\|F\|_{\mathrm{L}^p}^{p}}\geq\frac{2^{q k_2}}{\|G\|_{\mathrm{L}^q}^{q}},\,
\frac{2^{p k_1}}{\|F\|_{\mathrm{L}^p}^{p}}\geq\frac{2^{r k_3}}{\|H\|_{\mathrm{L}^r}^{r}}\}} $$
is controlled as
$$ \sum_{k_1\in\mathbb{Z}}\
\frac{2^{p k_1}}{\|F\|_{\mathrm{L}^p}^{p}}\Big(\!\sum_{Q\in\mathcal{M}_{k_1}^{F}} \!\!\!|Q|\Big)
\!\!\!\!\!\!\!\!
\sum_{{\scriptsize\begin{array}{c}k_2,k_3\in\mathbb{Z}\\(k_1,k_2,k_3)\in S_1\end{array}}}
\!\!\!\!\!\!\!\!
\bigg(\frac{{2^{q k_2}}/{\|G\|_{\mathrm{L}^q}^{q}}}{{2^{p k_1}}/{\|F\|_{\mathrm{L}^p}^{p}}}\bigg)^{\frac{1}{q}}
\bigg(\frac{{2^{r k_3}}/{\|H\|_{\mathrm{L}^r}^{r}}}{{2^{p k_1}}/{\|F\|_{\mathrm{L}^p}^{p}}}\bigg)^{\frac{1}{r}}
\lesssim_{p,q,r} 1 \,, $$
which follows from (\ref{vkeqsumtrees3}) and by summing two convergent geometric series
with their largest terms at most $1$, and ratios equal to $\frac{1}{2}$.

\section{Extending the range of exponents}
\label{vksectionextrange}

The extension of the main estimate to the range $p\leq 2$ or $q\leq 2$
follows from the conditional result of Bernicot, \cite{B}.
Here we repeat his argument in the dyadic case, where it is a bit simpler.
His idea is to use one-dimensional Calder\'{o}n-Zygmund decomposition in each fiber
$F(\cdot,y)$ or $G(x,\cdot)$.

We start with an estimate obtained in the previous section:
\begin{equation}
\label{vkeqextend0}
\|T_{\mathrm{d}}(F,G)\|_{\mathrm{L}^{pq/(p+q),\infty}} \leq\,\|T_{\mathrm{d}}(F,G)\|_{\mathrm{L}^{pq/(p+q)}}
\lesssim_{p,q} \|F\|_{\mathrm{L}^{p}} \|G\|_{\mathrm{L}^{q}} \,,
\end{equation}
for some $2<p,q<\infty$, $\frac{1}{p}+\frac{1}{q}>\frac{1}{2}$.
If we prove the weak estimate
\begin{align}
& \|T_{\mathrm{d}}(F,G)\|_{\mathrm{L}^{p/(p+1),\infty}} \lesssim_{p,q}
\|F\|_{\mathrm{L}^{p}} \|G\|_{\mathrm{L}^{1}} \,, \label{vkeqextend1}
\end{align}
then $T$ will be bounded in the whole range of Theorem \ref{vktheoremmainparaprod},
by real interpolation of multilinear operators, as stated for instance in \cite{GT} or \cite{T2}.
We first cover the part $p>2$, $q\leq 2$, then use (\ref{vkeqsymmetryintro}) for $p\leq 2$, $q>2$,
and finally repeat the argument to tackle the case $p,q\leq 2$.

By homogeneity we may assume \,$\|F\|_{\mathrm{L}^{p}}=\|G\|_{\mathrm{L}^{1}}=1$.
For each $x\in\mathbb{R}$ denote by $\mathcal{J}_x$
the collection of all maximal dyadic intervals $J$ with the property
$$ \frac{1}{|J|}\int_{J}|G(x,y)|\,dy > 1 \,. $$
Furthermore, set
$$ E:=\bigcup_{x\in\mathbb{R}}\bigcup_{J\in\mathcal{J}_x}(\{x\}\times J)  \,. $$
By our qualitative assumptions on $G$,
the set $E$ is simply a finite union of dyadic rectangles.
Using disjointness of $J\in\mathcal{J}_x$
\begin{equation}
\label{vkeqextend3}
|E| = \int_{\mathbb{R}} \sum_{J\in\mathcal{J}_x}|J| \,dx \leq
\int_{\mathbb{R}} \Big( \sum_{J\in\mathcal{J}_x}\int_{J}|G(x,y)| \,dy \Big) \,dx \leq 1 \,.
\end{equation}
Next, we define ``the good part'' of $G$ by
$$ \widetilde{G}(x,y) := \left\{\begin{array}{cl}
\frac{1}{|J|}\int_{J}G(x,v)dv, & \textrm{ for }y\in J\in\mathcal{J}_x \\
G(x,y), & \textrm{ for }(x,y)\not\in E
\end{array}\right. $$
By the construction of $\mathcal{J}_x$ we have
$\|\widetilde{G}\|_{\mathrm{L}^\infty}\leq 2$,
and from $\|\widetilde{G}\|_{\mathrm{L}^1}\leq 1$ we also get
$\|\widetilde{G}\|_{\mathrm{L}^q}\leq 2$,
so using the known estimate (\ref{vkeqextend0}) we obtain
\begin{equation}
\label{vkeqextend4}
\big|\big\{(x,y):|T_{\mathrm{d}}(F,\widetilde{G})(x,y)|>1\big\}\big| \ \lesssim_{p,q}\, 1 \,.
\end{equation}

As the last ingredient, we show that
\begin{equation}
\label{vkeqextend5}
\Big( \int_{\mathbb{R}}\big(G(x,v)\!-\!\widetilde{G}(x,v)\big) \,\psi^{\mathrm{d}}_{J'}(v) dv \Big) \,\psi^{\mathrm{d}}_{J'}(y) = 0
\end{equation}
for every $J'\in\mathcal{D}$, whenever $(x,y)\not\in E$.
Since $G(x,\cdot)-\widetilde{G}(x,\cdot)$ is supported on $\bigcup_{J\in\mathcal{J}_x}J$,
this in turn will follow from
\begin{equation}
\label{vkeqextend5ver2}
\Big( \int_{\mathbb{R}}\big(G(x,v)\!-\!\widetilde{G}(x,v)\big)
\,\psi^{\mathrm{d}}_{J'}(v) \,\mathbf{1}_{J}(v) \,dv \Big) \,\psi^{\mathrm{d}}_{J'}(y) = 0
\end{equation}
for every $J\in\mathcal{J}_x$.
In order to verify (\ref{vkeqextend5ver2}) it is enough to consider $J\cap J' \neq\emptyset$ and $y\in J'$,
and since $y\not\in J$, we conclude that $J$ is strictly contained in $J'$.
In this case $\psi^{\mathrm{d}}_{J'}(v) \,\mathbf{1}_{J}(v)=\pm |J'|^{-1/2}\mathbf{1}_{J}(v)$,
so we only have to observe
$\int_{J}\big(G(x,v)\!-\!\widetilde{G}(x,v)\big)dv=0$,
by the definition of $\widetilde{G}$.

Equation (\ref{vkeqextend5}) immediately gives $T_{\mathrm{d}}(F,G\!-\!\widetilde{G})(x,y)=0$ for $(x,y)\not\in E$, so
$$ \big\{(x,y):|T_{\mathrm{d}}(F,G)(x,y)|>1\big\} \subseteq E \cup
\big\{(x,y):|T_{\mathrm{d}}(F,\widetilde{G})(x,y)|>1\big\}, $$
and then from (\ref{vkeqextend3}) and (\ref{vkeqextend4})
$$ \big|\big\{(x,y):|T_{\mathrm{d}}(F,G)(x,y)|>1\big\}\big| \ \lesssim_{p,q}\, 1 \,. $$
This establishes (\ref{vkeqextend1}) by dyadic scaling.

\section{Transition to the continuous case}
\label{vksectionrealcase}

Now we turn to the task of proving strong estimates for $T_\mathrm{c}$ in the range from part (a)
of Theorem \ref{vktheoremmainparaprod}:
$$ \|T_{\mathrm{c}}(F,G)\|_{\mathrm{L}^{pq/(p+q)}} \lesssim_{p,q}
\|F\|_{\mathrm{L}^{p}} \|G\|_{\mathrm{L}^{q}} $$
for \,$1<p,q<\infty$,\, $\frac{1}{p}+\frac{1}{q}>\frac{1}{2}$.
In order to get the boundary weak estimates, one can later proceed as in \cite{B}.

Let $\varphi$ and $\psi$ be as in the introduction.
If $\int_{\mathbb{R}}\varphi=0$, then $T_\mathrm{c}(F,G)$ is dominated by
$$ \Big(\sum_{k\in\mathbb{Z}} |\mathrm{P}_{\varphi_k}^{(1)}F|^2 \Big)^{1/2}
\Big(\sum_{k\in\mathbb{Z}} |\mathrm{P}_{\psi_k}^{(2)}G|^2 \Big)^{1/2}, $$
and it is enough to use bounds for the two square functions.
Otherwise, we have
\,$0<|\int_{\mathbb{R}}\varphi| \lesssim 1$\,,
so let us normalize $\int_{\mathbb{R}}\varphi =1$.

A tool that comes in very handy here is the square function of Jones, Seeger, and Wright \cite{JSW}.
It effectively compares convolutions to martingale averages,
allowing us to do the transition easily.
\begin{proposition}[from \cite{JSW}]
\label{vkpropjsw}
Let $\varphi$ be a function satisfying \emph{(\ref{vkeqsymbolbounds})} and $\int_{\mathbb{R}}\varphi =1$.
The square function
$$ \mathcal{S}_{\mathrm{JSW},\varphi}f := \Big(\sum_{k\in\mathbb{Z}}
\big|\mathrm{P}_{\varphi_k}f - \mathbb{E}_{k}f\big|^2 \Big)^{1/2} $$
is bounded from $\mathrm{L}^p(\mathbb{R})$ to $\mathrm{L}^p(\mathbb{R})$ for $1<p<\infty$,
with the constant depending only\linebreak on $p$.
\end{proposition}

Let $\phi$ be a nonnegative $\mathrm{C}^\infty$ function such that
$\hat{\phi}(\xi)=1$ for $|\xi|\leq 2^{-0.6}$,
and $\hat{\phi}(\xi)=0$ for $|\xi|\geq 2^{-0.4}$.
We regard it as fixed, so we do not keep track of dependence of constants on $\phi$.
For any $a\in\mathbb{R}$ define $\phi_a$, $\vartheta_a$, $\rho_a$ by
\begin{align*}
\hat{\phi}_a(\xi) & :=\hat{\phi}(2^{-a}\xi) \,, \\
\hat{\vartheta}_a(\xi) & := \hat{\phi}(2^{-a-1}\xi)-\hat{\phi}(2^{-a}\xi)
\,=\, \hat{\phi}_{a+1}(\xi) - \hat{\phi}_a(\xi)  \,, \\
\hat{\rho}_a(\xi) & := \hat{\phi}(2^{-a-0.6}\xi)-\hat{\phi}(2^{-a-0.5}\xi) \,,
\end{align*}
so that in particular
\begin{align}
\hat{\vartheta}_{a}=1 & \quad\textrm{on }\mathrm{supp}(\hat{\rho}_{a}) \,, \label{vkeqspecialsupports1} \\
{\textstyle\sum_{i=-20}^{20}}\,\hat{\rho}_{k+0.1i}=1 & \quad\textrm{on }\mathrm{supp}(\hat{\psi}_{k}) \,,
\label{vkeqspecialsupports2} \\
{\textstyle\sum_{i=-20}^{20}}\,\hat{\rho}_{k+0.1i}=0 & \quad\textrm{on }\mathrm{supp}(\hat{\psi}_{k'})
\ \textrm{ if }|k'-k|\geq 10 \,. \label{vkeqspecialsupports3}
\end{align}

\smallskip
We first use Proposition \ref{vkpropjsw} to obtain bounds for a special case of
our continuous twisted paraproduct:
\begin{equation}
\label{vkeqtwspecialcase}
T_{\varphi,\vartheta,b}(F,G) :=
\sum_{k\in\mathbb{Z}} (\mathrm{P}_{\varphi_k}^{(1)}F) (\mathrm{P}_{\vartheta_{k+b}}^{(2)}G) \,,
\end{equation}
where $b\in\mathbb{R}$ is a fixed parameter. The constants can depend on $b$,
as later $b$ will take only finitely many concrete values.
Since we have already established estimates for (\ref{vkeqparaproductdyadic}),
it is enough to bound their difference:
\begin{equation}
\label{vkeqdiffestimate}
\big\| T_{\varphi,\vartheta,b}(F,G) - T_{\mathrm{d}}(F,G) \big\|_{\mathrm{L}^{pq/(p+q)}} \lesssim_{p,q,b}
\|F\|_{\mathrm{L}^{p}} \|G\|_{\mathrm{L}^{q}} \,.
\end{equation}
We introduce a mixed-type operator
$$ T_{\mathrm{aux},b}(F,G) := \sum_{k\in\mathbb{Z}}\, (\mathbb{E}_{k}^{(1)}F) (\mathrm{P}_{\vartheta_{k+b}}^{(2)}G) \,. $$
Using the Cauchy-Schwarz inequality in $k\in\mathbb{Z}$, one gets
$$ \big| T_{\varphi,\vartheta,b}(F,G) - T_{\mathrm{aux},b}(F,G)\big| \,\leq\,
\Big( \sum_{k\in\mathbb{Z}} \big|\mathrm{P}_{\varphi_k}^{(1)}F - \mathbb{E}_{k}^{(1)}F\big|^2 \Big)^{1/2}
\Big( \sum_{k\in\mathbb{Z}} \big|\mathrm{P}_{\vartheta_{k+b}}^{(2)}G\big|^2 \Big)^{1/2} \,. $$
The first term on the right hand side is $\mathcal{S}_{\mathrm{JSW},\varphi}^{(1)}F$,
while the second one is the ordinary square function in the second variable, as $\int_\mathbb{R}\vartheta_b=0$.
Next, one can rewrite $T_{\mathrm{aux},b}$ and $T_{\mathrm{d}}$ as
\begin{align*}
T_{\mathrm{aux},b}(F,G) & = FG - \sum_{k\in\mathbb{Z}}\, (\Delta_{k}^{(1)}F) (\mathrm{P}_{\phi_{k+1+b}}^{(2)}G) \,, \\
T_{\mathrm{d}}(F,G) & = FG - \sum_{k\in\mathbb{Z}}\, (\Delta_{k}^{(1)}F) (\mathbb{E}_{k+1}^{(2)}G) \,.
\end{align*}
Subtracting and using the Cauchy-Schwarz inequality in $k\in\mathbb{Z}$, this time we obtain
$$ \big|T_{\mathrm{aux},b}(F,G) - T_{\mathrm{d}}(F,G)\big| \,\leq\,
\Big( \sum_{k\in\mathbb{Z}} \big|\Delta_{k}^{(1)}F\big|^2 \Big)^{1/2}
\Big( \sum_{k\in\mathbb{Z}} \big|\mathrm{P}_{\phi_{k+b}}^{(2)}G - \mathbb{E}_{k}^{(2)}G\big|^2 \Big)^{1/2} \,. $$
The first term on the right hand side is just the dyadic square function in the first variable,
while the second term is $\mathcal{S}_{\mathrm{JSW},\phi_b}^{(2)}G$.
The estimate (\ref{vkeqdiffestimate}) now follows from Proposition \ref{vkpropjsw}
and bounds on the two common square functions.

\smallskip
Actually, we need a ``sparser'' paraproduct than the one in (\ref{vkeqtwspecialcase}):
\begin{equation}
\label{vkeqtwspecialsparse}
T_{\varphi,\rho,b,l}^{10\mathbb{Z}}(F,G) :=
\sum_{j\in\mathbb{Z}} (\mathrm{P}_{\varphi_{10j+l}}^{(1)}F) (\mathrm{P}_{\rho_{10j+l+b}}^{(2)}G) \,,
\end{equation}
for any $l=0,1,\ldots,9$.
To see that (\ref{vkeqtwspecialsparse}) is bounded too,
we define
$$ \widetilde{G}_{b,l} := \sum_{j\in\mathbb{Z}} \mathrm{P}_{\rho_{10j+l+b}}^{(2)}G \,. $$
Notice that because of (\ref{vkeqspecialsupports1}) we have
$$ \mathrm{P}_{\vartheta_{k+b}}^{(2)}\widetilde{G}_{b,l} =
\left\{\begin{array}{cl}\mathrm{P}_{\rho_{10j+l+b}}^{(2)}G & \textrm{ for }k=10j+l\in 10\mathbb{Z}+l \\
0 & \textrm{ for }k\in\mathbb{Z},\ k\not\in 10\mathbb{Z}+l \end{array}\right. $$
and the Littlewood-Paley inequality gives
$$ \|\widetilde{G}_{b,l}\|_{\mathrm{L}^q} \lesssim_{q,b,l} \|G\|_{\mathrm{L}^q} \,. $$
It remains to write
$$ T_{\varphi,\rho,b,l}^{10\mathbb{Z}}(F,G) = T_{\varphi,\vartheta,b}(F,\widetilde{G}_{b,l}) \,, $$
and use boundedness of (\ref{vkeqtwspecialcase}).

\smallskip
Finally, we tackle the original operator (\ref{vkeqparaproductreal}).
The following computation is possible
because of (\ref{vkeqspecialsupports2}) and (\ref{vkeqspecialsupports3}).
\begin{align*}
\sum_{k\in\mathbb{Z}} \hat{\varphi}_{k}(\xi) \hat{\psi}_{k}(\eta)
& = \sum_{l=0}^{9} \sum_{j\in\mathbb{Z}} \hat{\varphi}_{10j+l}(\xi) \hat{\psi}_{10j+l}(\eta) \\
& = \sum_{l=0}^{9} \sum_{i=-20}^{20} \sum_{j\in\mathbb{Z}}
\hat{\varphi}_{10j+l}(\xi) \hat{\rho}_{10j+l+0.1i}(\eta) \hat{\psi}_{10j+l}(\eta) \\
& = \sum_{l=0}^{9} \sum_{i=-20}^{20} \sum_{j\in\mathbb{Z}}
\hat{\varphi}_{10j+l}(\xi) \hat{\rho}_{10j+l+0.1i}(\eta) \hat{\Psi}_{l}(\eta)
\end{align*}
Above we have set \,$\Psi_l := \sum_{m\in\mathbb{Z}} \psi_{10m+l}$.
This ``symbol identity'' leads us to
\begin{equation}
\label{vkeqcontfinal}
T_{\mathrm{c}}(F,G) = \sum_{l=0}^{9} \sum_{i=-20}^{20}
T_{\varphi,\,\rho,\,0.1i,\,l}^{10\mathbb{Z}}(F,\mathrm{P}_{\Psi_l}^{(2)}G) \,. \\
\end{equation}
Since $\hat{\psi}$ has a compact support and
$|\hat{\psi}(\eta)|,\,|\frac{d}{d\eta}\hat{\psi}(\eta)|\lesssim 1$
by (\ref{vkeqsymbolbounds}),
scaling gives
\,$|\hat{\Psi}_l(\eta)|\lesssim 1$,\, $\big|\frac{d}{d\eta}\hat{\Psi}_l(\eta)\big|\lesssim |\eta|^{-1}$,\,
and thus the H\"{o}rmander-Mikhlin multiplier theorem (in one variable) implies
$$ \big\|\mathrm{P}_{\Psi_{l}}^{(2)} G\big\|_{\mathrm{L}^q} \lesssim_{q,l} \|G\|_{\mathrm{L}^{q}} \,. $$
It remains to use (\ref{vkeqcontfinal}) and boundedness of (\ref{vkeqtwspecialsparse}).

\section{Endpoint counterexamples}
\label{vksectioncounterex}

We give the arguments in the dyadic setting, the continuous case being similar.
First we show that $T_\mathrm{d}$ does not map boundedly
$$ \mathrm{L}^\infty(\mathbb{R}^2)\times\mathrm{L}^q(\mathbb{R}^2) \to \mathrm{L}^{q,\infty}(\mathbb{R}^2) $$
for $1\leq q<\infty$.
Take $G$ to be
$$ G(x,y) := \mathbf{1}_{[0,2^{-n})}(x) \sum_{k=0}^{n-1} R_{k+1}(y) \,, $$
for some positive integer $n$, where $R_k$ denotes the $k$-th Rademacher function\footnote{
Linear combinations of Rademacher functions $\sum_k c_k R_k(t)$
are dyadic analogues of lacunary trigonometric series $\sum_k c_k e^{i 2^k t}$.}
on $[0,1)$, \ i.e.
$$ R_k := \!\sum_{J\subseteq[0,1),\, |J|=2^{-k+1}} \!(\mathbf{1}_{J_\mathrm{left}}-\mathbf{1}_{J_\mathrm{right}}) \,. $$
Recall Khintchine's inequality, which can be formulated as:
$$ \Big\|\sum_{k=1}^{n} c_k R_k\Big\|_{\mathrm{L}^q} \sim_q
\Big(\sum_{k=1}^{n}|c_k|^2\Big)^{1/2},\qquad\textrm{for } 0<q<\infty \,, $$
giving us \,$\|G\|_{\mathrm{L}^q} \sim_q\, 2^{-n/q}n^{1/2}$.\,
Observe that \,$(\Delta_{k}^{(2)}G)(x,y)=\mathbf{1}_{[0,2^{-n})}(x)R_{k+1}(y)$\,
for $k=0,1,\ldots,n-1$.

We choose $F$ supported in the unit square $[0,1)^2$ and defined by
$$ F(x,y) := \left\{\begin{array}{cl}
2R_{j}(y) - R_{j+1}(y), & \textrm{for }x\in[2^{-j},2^{-j+1}),\ j=1,\ldots,n\!-\!1 \\
R_{n}(y), & \textrm{for }x\in[0,2^{-n+1})
\end{array}\right. $$
Note that \,$\|F\|_{\mathrm{L}^\infty}\leq 3$\, and
\,$(\mathbb{E}_{k}^{(1)}F)(x,y)=R_{k+1}(y)$\, for $x\in[0,2^{-n})$, $k=0,1,\ldots,n-1$.
Since the output function is now simply
$T_{\mathrm{d}}(F,G) = n\,\mathbf{1}_{[0,2^{-n})\times[0,1)}$, we have
$$ \frac{\|T_{\mathrm{d}}(F,G)\|_{\mathrm{L}^{q,\infty}}}{\|F\|_{\mathrm{L}^\infty} \|G\|_{\mathrm{L}^q}}
\,\gtrsim_q \frac{2^{-n/q}n}{2^{-n/q}n^{1/2}} = n^{1/2} \,, $$
which shows unboundedness.

\medskip
The remaining estimate
\,$\|T_{\mathrm{d}}(F,G)\|_{\mathrm{L}^{\infty}} \lesssim \|F\|_{\mathrm{L}^{\infty}} \|G\|_{\mathrm{L}^{\infty}}$\,
is even easier to disprove.
For a positive integer $n$, take
$$ F(x,y) := \left\{\begin{array}{cl}
1, & \textrm{for }\,x\in\bigcup_{j=0}^{n-1}[2^{-2j-1}\!,2^{-2j}), \ y\in[0,1) \\
0, & \textrm{otherwise}
\end{array}\right. $$
and $G(x,y):=F(y,x)$. It is easy to see that
\,$|T_{\mathrm{d}}(F,G)(x,y)| \sim n$\, on the square $(x,y)\in[0,2^{-2n})^2$.

\section{Closing remarks}
\label{vksectionclosing}

The decomposition into trees from Section \ref{vksectionsummingtrees}
has its primary purpose in proving the estimate for a larger range of exponents.
If one is content with just having estimates in some nontrivial range,
then a simpler proof can be given.
Using Lemma \ref{vklemmareduction}:
\begin{align}
|\Lambda_{\mathrm{d}}(F,G,H)| & \,\leq\, \Theta_{\mathcal{C}}^{(2)}(\mathbf{1},G,\mathbf{1},G)^{1/2}
\,\Theta_{\mathcal{C}}^{(2)}(F,H,F,H)^{1/2} \,, \label{vkeqclosingsimple1} \\
\big|\Theta_{\mathcal{C}}^{(1)}(F,H,F,H)\big| & \,\leq\,
\Theta_{\mathcal{C}}^{(1)}(F,F,F,F)^{1/2} \,\Theta_{\mathcal{C}}^{(1)}(H,H,H,H)^{1/2} \,.
\end{align}
If in Lemma \ref{vklemmatelescoping} one lets a single tree $\mathcal{T}$ exhaust the family
of all dyadic squares, then the telescoping identity becomes simply
$$ \Theta_{\mathcal{C}}^{(1)}(F_1,F_2,F_3,F_4) + \Theta_{\mathcal{C}}^{(2)}(F_1,F_2,F_3,F_4)
= \int_{\mathbb{R}^2}\! F_1 F_2 F_3 F_4 \,. $$
Particular instances of this equality are:
\begin{align}
\Theta_{\mathcal{C}}^{(2)}(F,H,F,H) & = \|FH\|_{\mathrm{L}^2}^2 - \Theta_{\mathcal{C}}^{(1)}(F,H,F,H) \,, \\
\Theta_{\mathcal{C}}^{(2)}(\mathbf{1},G,\mathbf{1},G) & = \|G\|_{\mathrm{L}^2}^2
- \Theta_{\mathcal{C}}^{(1)}(\mathbf{1},G,\mathbf{1},G) = \|G\|_{\mathrm{L}^2}^2 \,, \\
\Theta_{\mathcal{C}}^{(1)}(F,F,F,F) & = \|F\|_{\mathrm{L}^4}^4
- \Theta_{\mathcal{C}}^{(2)}(F,F,F,F) \leq \|F\|_{\mathrm{L}^4}^4 \,, \\
\Theta_{\mathcal{C}}^{(1)}(H,H,H,H) & = \|H\|_{\mathrm{L}^4}^4
- \Theta_{\mathcal{C}}^{(2)}(H,H,H,H) \leq \|H\|_{\mathrm{L}^4}^4 \,. \label{vkeqclosingsimple2}
\end{align}
Combining (\ref{vkeqclosingsimple1})--(\ref{vkeqclosingsimple2}) one ends up with
$$ |\Lambda_{\mathrm{d}}(F,G,H)| \leq \|G\|_{\mathrm{L}^2}
\Big( \|FH\|_{\mathrm{L}^2}^2 + \|F\|_{\mathrm{L}^4}^2 \|H\|_{\mathrm{L}^4}^2 \Big)^{1/2} , $$
which establishes the estimate for $(p,q,r)=(4,2,4)$.
By symmetry one also gets the point $(p,q,r)=(2,4,4)$, and then uses interpolation
and the method from Section \ref{vksectionextrange}.
However, that way we would leave out the larger part of the Banach triangle,
including the ``central'' point $(p,q,r)=(3,3,3)$.

\bigskip
Starting from the single tree estimate (\ref{vkeqsingletreetheta})
and adjusting the arguments from Section \ref{vksectionsummingtrees} in the obvious way,
we also obtain estimates for an even more ``entwined'' form:
\begin{align*}
\Theta_{\mathcal{C}}^{(2)}(F_1,F_2,F_3,F_4)
= \sum_{I\times J\in\mathcal{C}}
\int_{\mathbb{R}^4}\! F_1(u,v) F_2(x,v) F_3(u,y) F_4(x,y) \qquad & \\[-2mm]
\varphi^{\mathrm{d}}_I(u)\varphi^{\mathrm{d}}_I(x)\psi^{\mathrm{d}}_J(v)\psi^{\mathrm{d}}_J(y) \ du dv dx dy & \,.
\end{align*}
The bound we get is
$$ \big|\Theta_{\mathcal{C}}^{(2)}(F_1,F_2,F_3,F_4)\big| \ \lesssim_{p_1,p_2,p_3,p_4}\
\prod_{j=1}^{4}\|F_j\|_{\mathrm{L}^{p_j}}, $$
whenever \,$\frac{1}{p_1}\!+\!\frac{1}{p_2}\!+\!\frac{1}{p_3}\!+\!\frac{1}{p_4}=1$,\ \
$2<p_1,p_2,p_3,p_4<\infty$.\,
This time we do not know of any arguments from the Calder\'{o}n-Zygmund theory that could
help expand the range of exponents.

\bigskip
Let us conclude with several words on a straightforward generalization of the method
presented in Sections \ref{vksectiontelescoping} and \ref{vksectionsummingtrees}
to higher dimensions.
For notational simplicity we only state the result in $\mathbb{R}^3$.
\begin{theorem}
\label{vktheoremgeneral}
For any $S\subseteq \{0,1,\ldots,7\}$ we define a multilinear form $\Lambda_S$,
acting on $|S|$ functions $F_{j}\colon\mathbb{R}^3\to\mathbb{C}$ by
\begin{align*}
\Lambda_S \big((F_{j})_{j\in S}\big)
& := \sum_{Q}\,\int_{\mathbb{R}^{6}}
\prod_{j\in S} F_{j}\big(x_{1}^{j_1},x_{2}^{j_2},x_{3}^{j_3}\big)
\ \varphi^{\mathrm{d}}_{I_1}\!(x_1^{0})\varphi^{\mathrm{d}}_{I_1}\!(x_1^{1})  \\
& \varphi^{\mathrm{d}}_{I_2}\!(x_2^{0})\varphi^{\mathrm{d}}_{I_2}\!(x_2^{1})
\,\psi^{\mathrm{d}}_{I_3}\!(x_3^{0})\psi^{\mathrm{d}}_{I_3}\!(x_3^{1})
\, dx_1^{0}dx_1^{1}dx_2^{0}dx_2^{1}dx_3^{0}dx_3^{1}  \,,
\end{align*}
where \,$Q=I_1\times I_2\times I_3$\, is a dyadic cube, and \,$j=j_1+2j_2+4j_3$,\, $j_1,j_2,j_3\in\{0,1\}$.
Then $\Lambda_S$ satisfies the bound
$$ \big|\Lambda_S \big((F_{j})_{j\in S}\big)\big| \ \ \lesssim_{(p_{j})_{j\in S}} \
\prod_{j\in S} \|F_{j}\|_{\mathrm{L}^{p_j}(\mathbb{R}^3)} \,, $$
whenever the exponents $(p_j)_{j\in S}$ are such that
\,$\sum_{j\in S} \frac{1}{p_j} =1$,\, and \,$4<p_j<\infty$\,
for every \,$j\in S$.
\end{theorem}
The result is nontrivial only when $|S|\geq 5$.
We sketch a proof of Theorem \ref{vktheoremgeneral}, which uses the same ingredients as before.

Dyadic cubes $Q=I_1\times I_2\times I_3\subseteq\mathbb{R}^3$
are again organized into families of trees. For each tree $\mathcal{T}$ we define
the three local forms $\Theta_{\mathcal{T}}^{(1)}$\!,
$\Theta_{\mathcal{T}}^{(2)}$\!, $\Theta_{\mathcal{T}}^{(3)}$. For instance
\begin{align*}
\Theta_{\mathcal{T}}^{(1)}(F_0,\ldots,F_7)
& := \sum_{Q\in\mathcal{T}}\,\int_{\mathbb{R}^{6}}
\prod_{j=0}^{7} F_{j}\big(x_{1}^{j_1},x_{2}^{j_2},x_{3}^{j_3}\big)
\!\!\sum_{\alpha,\beta\in\{\mathrm{left},\mathrm{right}\}}\!\!\!\!
\psi^{\mathrm{d}}_{I_1}\!(x_1^{0})\psi^{\mathrm{d}}_{I_1}\!(x_1^{1})  \\
& \varphi^{\mathrm{d}}_{I_{2,\alpha}}\!(x_2^{0})\varphi^{\mathrm{d}}_{I_{2,\alpha}}\!(x_2^{1})
\,\varphi^{\mathrm{d}}_{I_{3,\beta}}\!(x_3^{0})\varphi^{\mathrm{d}}_{I_{3,\beta}}\!(x_3^{1})
\ dx_1^{0}dx_1^{1}dx_2^{0}dx_2^{1}dx_3^{0}dx_3^{1} \,.
\end{align*}
The form $\Xi_\mathcal{F}$ is defined analogously, with $[\cdot]_{\Box(Q)}$
replaced by the three-dimensional Gowers box inner-product:
$$ \left[F_0,\ldots,F_7\right]_{\Box^3(Q)}:=
\mathbb{E}\,\Big( \prod_{j=0}^{7} F_{j}\big(x_{1}^{j_1},x_{2}^{j_2},x_{3}^{j_3}\big)
\ \Big| \ x_1^{0},x_1^{1}\!\in\! I_1,\, x_2^{0},x_2^{1}\!\in\! I_2,\, x_3^{0},x_3^{1}\!\in\! I_3 \Big), $$
in the probabilistic notation.
However, Inequality (\ref{vkeqboxlessl2}) has to be replaced with
$$ \|F\|_{\Box^3(Q)} \,\leq\, \Big(\frac{1}{|Q|}\int_Q |F|^{4} \Big)^{1/4} , $$
which is the reason why the range of exponents is severely restricted.

The telescoping identity now has three terms on the left hand side:
\begin{equation}
\label{vkeqgentele}
\Theta_{\mathcal{T}}^{(1)} + \Theta_{\mathcal{T}}^{(2)} + \Theta_{\mathcal{T}}^{(3)}
= \Xi_{\mathcal{L}(\mathcal{T})} - \Xi_{\{Q_\mathcal{T}\}} \,.
\end{equation}
The proof of the single tree estimate is inductive,
with alternating applications of Identity (\ref{vkeqgentele}) and the Cauchy-Schwarz inequality.
The telescoping identity reduces the problem of controlling a particular ``theta-term'',
$\Theta^{(i)}(F_{k_0},\ldots,F_{k_7})$, to bounding two other theta-terms.
If any of the latter ones is nonnegative, it can be ignored.
On the other hand, any term that is not nonnegative can be estimated, using an analogue of Lemma \ref{vklemmareduction},
by two nonnegative terms with smaller number of different functions involved.
The induction starts with theta-terms containing only one function, $\Theta^{(i)}(F_{k},\ldots,F_{k})$,
but these are obviously nonnegative.

Figure \ref{vkdiagram3d} presents these steps in the form of a tree-diagram.
We draw only essentially different branches, i.e.\@ omit the ones that can be treated by analogy.
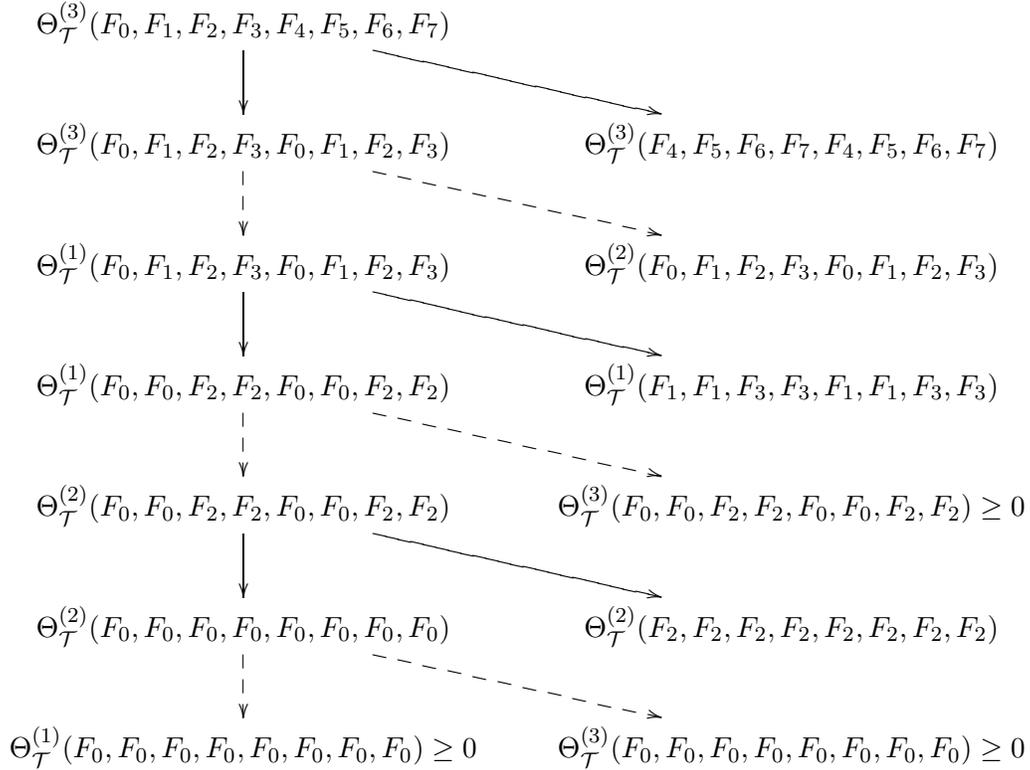
\begin{figure}[t]
{\small $$ \xymatrix {
\Theta_{\mathcal{T}}^{(3)}(F_0,F_1,F_2,F_3,F_4,F_5,F_6,F_7) \ar@{->}[d] \ar@{->}[dr] & \\
\Theta_{\mathcal{T}}^{(3)}(F_0,F_1,F_2,F_3,F_0,F_1,F_2,F_3) \ar@{-->}[d] \ar@{-->}[dr]
& \Theta_{\mathcal{T}}^{(3)}(F_4,F_5,F_6,F_7,F_4,F_5,F_6,F_7) \\
\Theta_{\mathcal{T}}^{(1)}(F_0,F_1,F_2,F_3,F_0,F_1,F_2,F_3) \ar@{->}[d] \ar@{->}[dr]
& \Theta_{\mathcal{T}}^{(2)}(F_0,F_1,F_2,F_3,F_0,F_1,F_2,F_3) \\
\Theta_{\mathcal{T}}^{(1)}(F_0,F_0,F_2,F_2,F_0,F_0,F_2,F_2) \ar@{-->}[d] \ar@{-->}[dr]
& \Theta_{\mathcal{T}}^{(1)}(F_1,F_1,F_3,F_3,F_1,F_1,F_3,F_3) \\
\Theta_{\mathcal{T}}^{(2)}(F_0,F_0,F_2,F_2,F_0,F_0,F_2,F_2) \ar@{->}[d] \ar@{->}[dr]
& \Theta_{\mathcal{T}}^{(3)}(F_0,F_0,F_2,F_2,F_0,F_0,F_2,F_2)\geq 0 \\
\Theta_{\mathcal{T}}^{(2)}(F_0,F_0,F_0,F_0,F_0,F_0,F_0,F_0) \ar@{-->}[d] \ar@{-->}[dr]
& \Theta_{\mathcal{T}}^{(2)}(F_2,F_2,F_2,F_2,F_2,F_2,F_2,F_2) \\
\Theta_{\mathcal{T}}^{(1)}(F_0,F_0,F_0,F_0,F_0,F_0,F_0,F_0)\geq 0
& \Theta_{\mathcal{T}}^{(3)}(F_0,F_0,F_0,F_0,F_0,F_0,F_0,F_0)\geq 0
} $$ }
\caption{The proof of the single tree estimate in $\mathbb{R}^3$.
A solid arrow denotes an application of the Cauchy-Schwarz inequality,
while a broken arrow denotes an application of Identity (\ref{vkeqgentele}).}
\label{vkdiagram3d}
\end{figure}

\bibliographystyle{amsplain}

\end{document}